\documentclass[11pt]{amsart}

\usepackage{geometry}
\allowdisplaybreaks

\usepackage[dvipsnames]{xcolor}
\usepackage[colorlinks,citecolor=OliveGreen,linkcolor=Mahogany,urlcolor=Plum,pagebackref]{hyperref}
\usepackage[alphabetic]{amsrefs}

\usepackage{amsfonts,amssymb,amscd}
\usepackage{mathtools}
\numberwithin{equation}{section}

\newtheorem{theorem}{Theorem}[section]
\newtheorem{lemma}[theorem]{Lemma}
\newtheorem{proposition}[theorem]{Proposition}
\newtheorem{corollary}[theorem]{Corollary}
\theoremstyle{definition}
\newtheorem{definition}[theorem]{Definition}

\newtheorem{question}[theorem]{Question}

\newtheorem{remark}[theorem]{Remark}

\newcommand{\mld}{{\operatorname{mld}}}
\newcommand{\vol}{{\operatorname{vol}}}
\newcommand{\volhat}{\widehat{\operatorname{vol}}}
\newcommand{\ord}{{\operatorname{ord}}}
\newcommand{\ind}{{\operatorname{ind}}}
\newcommand{\Val}{{\operatorname{Val}}}
\newcommand{\Spec}{{\operatorname{Spec}}}
\newcommand{\Supp}{{\operatorname{Supp}}}
\newcommand{\C}{\mathbb C}
\newcommand{\Q}{\mathbb Q}
\newcommand{\R}{\mathbb R}
\newcommand{\Z}{\mathbb Z}

\begin{document}

\title{A Sharp Inequality between Local Volumes and Minimal Log Discrepancies}
\author{Jingjun Han}

\begin{abstract}
We answer a question of Li--Liu--Xu:
every $n$-dimensional klt germ $x\in(X,\Delta)$, where $n\ge2$,
satisfies the sharp inequality
\[
  \widehat{\operatorname{vol}}(x,X,\Delta)\le n^{n-1}\operatorname{mld}_x(X,\Delta),
\]
with equality if and only if $\Delta=0$ near $x$, and analytically,
$(x\in X)\cong\frac{1}{r}(1,\ldots,1)$ for some $r\ge1$.  We also prove
that, in fixed dimension and with coefficients in a fixed finite set,
$\volhat/\mld$ is discrete away from zero.  As applications of the sharp
inequality, we obtain lower bounds for minimal log discrepancies of log
Fano pairs.
\end{abstract}

\address{Shanghai Center for Mathematical Sciences \& School of Mathematical Sciences, Fudan University, Shanghai 200438, China}
\email{hanjingjun@fudan.edu.cn}

\maketitle

\section{Introduction}

We work over the field of complex numbers $\C$.

The normalized volume was introduced by C.~Li \cite{Li18} with
motivation from K-stability.  Its minimum, the local volume of a klt germ, is
central to the local K-stability theory for klt singularities.  In particular, a
minimizer induces a K-semistable log Fano cone degeneration of the germ
\cite[Theorem~1.2]{XZ25}. Local volumes also connect singularities to global
Fano geometry.  Liu proved that the local volume at every closed point of a K-semistable Fano variety is bounded from below in terms of its anticanonical volume \cite[Theorem~2]{Liu18}.  Thus an upper bound for the local volume in terms of a classical singularity invariant gives uniform restrictions on the singularities of such Fano varieties.

Minimal log discrepancies are among the basic invariants of singularities in
the minimal model program.  The ACC conjecture for mlds is closely related to the
termination of flips and the existence of minimal models; see, for example,
\cites{Sho04,HL25}.  To compare local volume with minimal log
discrepancy, Li--Liu--Xu proved
\[
  \volhat(x,X,\Delta)<n^n\mld_x(X,\Delta)
\]
\cite[Theorem~6.13]{LLX20}.  Combined with \cite[Theorem~2]{Liu18},
this gives a lower bound for the minimal log discrepancy and leads to the boundedness of K-semistable log Fano varieties \cite[Corollary~6.14]{LLX20}.  They asked whether, when $\Delta=0$, the
coefficient $n^n$ can be replaced by $n^{n-1}$
\cite[Question~6.16]{LLX20}.  Liu proved this in dimension three
\cite[Theorem~1.4]{Liu25} and asked whether the same inequality holds for klt
threefold pairs \cite[Question~5.5]{Liu25}.

Our main theorem answers \cite[Question~6.16]{LLX20} and \cite[Question~5.5]{Liu25}.

\begin{theorem}\label{thm:main}
Let $x\in(X,\Delta)$ be an $n$-dimensional klt germ, where $n\ge2$.  Then
\begin{equation}\label{eq:main}
  \volhat(x,X,\Delta)
  \le n^{n-1}\mld_x(X,\Delta),
\end{equation}
with equality if and only if $\Delta=0$ near $x$, and analytically, $(x\in X)\cong\frac{1}{r}(1,\ldots,1)$ for some integer $r\ge1$.
\end{theorem}

Theorem~\ref{thm:main} has the following corollary. 

\begin{corollary}\label{cor:mld-one}
Let $x\in(X,\Delta)$ be an $n$-dimensional klt germ, where $n\ge2$, and
assume that $\mld_x(X,\Delta)\le1$.  Then
\[
  \volhat(x,X,\Delta)\le n^{n-1},
\]
with equality if and only if $\Delta=0$ near $x$, and analytically, $(x\in X)\cong\frac{1}{n}(1,\ldots,1)$.
\end{corollary}

For a Gorenstein canonical non-hypersurface threefold singularity, the
classification in \cite[Theorems~5.34--5.35]{KM98}, as recalled in
\cite[Theorem~2.19]{Liu25}, gives a prime divisor $E$ over $x\in X$ such that
$A_X(E)=1$.  Since $x\in X$ is canonical, $A_X(F)\ge1$ for every prime divisor
$F$ over $X$ centered at $x$.  Hence
$\mld_x(X)=1$, and Corollary~\ref{cor:mld-one} recovers \cite[Theorem~1.1]{Liu25}, including the equality
case $\frac{1}{3}(1,1,1)$.

Local volumes $\volhat(x,X)$ in fixed dimension are bounded above by $n^n$ \cite[Lemma~A.1]{LX19} and are discrete away from zero
\cite[Theorem~1.2]{XZ24}.  By Theorem~\ref{thm:main},
$\frac{\volhat(x,X,\Delta)}{\mld_x(X,\Delta)}\le n^{n-1}$.  We next study the possible values of these quotients.  For an integer $n\ge2$ and a set
$\Gamma\subset[0,1]$, set
\[
  \mathcal R_{n,\Gamma}
  :=\left\{
  \frac{\volhat(x,X,\Delta)}{\mld_x(X,\Delta)}
  \ \middle|\ 
  \begin{array}{c}
  x\in(X,\Delta)\text{ is an }n\text{-dimensional klt germ},\\[-2pt]
  \operatorname{Coeff}(\Delta)\subseteq\Gamma
  \end{array}
  \right\}.
\]

\begin{theorem}\label{thm:ratio-acc}
Fix an integer $n\ge2$ and a finite set $\Gamma\subset[0,1]$.  For every $\varepsilon>0$, the set
\[
  \mathcal R_{n,\Gamma}\cap[\varepsilon,\infty)
\]
is finite.  In particular, $\mathcal R_{n,\Gamma}$ is discrete away from zero and satisfies the ACC.
\end{theorem}

More generally, Proposition~\ref{prop:acc-threshold} proves the same
finiteness for $\volhat/\mld^c$ for every $0<c\le1$, whereas the ACC
fails for $c>1$.

Note that the desired finiteness is not a formal
consequence of the discreteness of local volumes: minimal log discrepancies may approach zero, and taking quotients may create new accumulation behavior.

\medskip
\noindent\emph{Applications.}
For a log Fano pair, the local-to-global estimate recalled in
Lemma~\ref{lem:local-global} gives
\[
  \volhat(x,X,\Delta)
  \ge
  \left(\frac{n}{n+1}\right)^n
  \delta(X,\Delta)^n
  \bigl(-(K_X+\Delta)\bigr)^n
\]
at every closed point, where $\delta(X,\Delta)$ denotes the $\delta$-invariant.  Combining this estimate with Theorem~\ref{thm:main} gives the following consequence.

\begin{corollary}\label{cor:kmoduli}
Let $(X,\Delta)$ be an $n$-dimensional log Fano pair, where $n\ge2$.  Then every closed point $x\in X$ satisfies
\begin{equation}\label{eq:kmoduli-mld}
  \mld_x(X,\Delta)
  \ge
  \frac{n}{(n+1)^n}\delta(X,\Delta)^n
  \bigl(-(K_X+\Delta)\bigr)^n.
\end{equation}
In particular, if $(X,\Delta)$ is K-semistable, then
\[
  \mld_x(X,\Delta)
  \ge
  \frac{n}{(n+1)^n}\bigl(-(K_X+\Delta)\bigr)^n.
\]
\end{corollary}

The coefficient $\frac{n}{(n+1)^n}$ in Corollary~\ref{cor:kmoduli} is optimal,
as equality holds for $(X,\Delta)=(\mathbb P^n,0)$ at every closed
point.

Thus, after fixing the dimension and anticanonical volume and imposing a
positive lower bound for the $\delta$-invariant, the minimal log
discrepancies at closed points have a uniform positive lower bound.  This
coefficient is $n$ times the coefficient obtained by combining the
Li--Liu--Xu inequality with the same local-to-global estimate.  The
threefold results in Section~\ref{sec:applications} illustrate how this general mld
bound can be combined with dimension-specific classifications from
\cite[Theorem~1.2 and Propositions~4.3--4.4]{Liu25}.  They also give bounds
for the boundary; see
Corollaries~\ref{cor:threefold-boundary} and~\ref{cor:boundary-restriction}.

\medskip
\noindent\emph{Acknowledgments.}
JH would like to thank Guodu Chen, Chen Jiang, Yuchen Liu, Yujie Luo, Lu Qi, Lingyao Xie, and Chenyang Xu, Ziquan Zhuang for fruitful discussions. This work was supported by the National Key R\&D Program
of China (No.~2025YFA1018100, No.~2023YFA1010600) and the NSFC for Excellent Young Scientists (No. 12322102). JH is a member of LMNS,
Fudan University. JH thanks ChatGPT 5.6 Pro for suggesting the semicontinuity examples in
Remark~\ref{rem:semicontinuity} and for assistance with English editing.

\section{Preliminaries}\label{sec:preliminaries}

\subsection{Local invariants and minimizers}

We follow the standard notation and conventions of \cite{KM98} for
singularities of pairs.

\begin{definition}[Pairs and germs]\label{def:pairs-germs}
A \emph{pair} $(X,\Delta)$ consists of a normal quasi-projective variety $X$ together with an $\R$-divisor $\Delta\ge 0$ such that $K_X+\Delta$ is $\R$-Cartier. If the coefficients of $\Delta$ are $\leq1$, then $\Delta$ is called a boundary on $X$.

For a prime divisor $F$ over $X$, we write $A_{X,\Delta}(F)$ for its log discrepancy.  A germ $x\in(X,\Delta)$ consists of a pair and a closed point $x\in X$; it is klt if the pair is klt in a neighborhood of $x$.  We say that a prime divisor $F$ is over $X\ni x$ if $F$ is over $X$ and its center on $X$ is $x$.
\end{definition}

Let $\Val_X$ denote the set of real valuations of $K(X)$ that are
trivial on $\C$ and admit a center on $X$.  For $v\in\Val_X$, we
denote its center by $c_X(v)$.  If $x\in X$ is a closed point, set
\[
  \Val_{X,x}:=\{v\in\Val_X\mid c_X(v)=x\}.
\]
We recall the
extension of log discrepancy to this space from \cite[Sections~4--5]{JM12}
and \cite{BdFFU15}.  A log smooth model $\mu\colon(Y,D)\to X$ of
$(X,\Delta)$ consists of a proper birational morphism from a smooth variety
$Y$ and a reduced simple normal crossing divisor $D$ containing the
exceptional locus of $\mu$ and the support of the strict transform of
$\Delta$.  We write $\operatorname{QM}(Y,D)$ for the quasi-monomial
valuations determined by the strata of $D$.  There is a canonical retraction $r_{Y,D}\colon\Val_X\to\operatorname{QM}(Y,D)$ as in \cite[Section~4.3]{JM12}.

\begin{definition}[Log discrepancy of valuations]
For a divisorial valuation $v=\lambda\ord_E$, set
$A_{X,\Delta}(v):=\lambda A_{X,\Delta}(E)$.  Let $v_\alpha\in\operatorname{QM}(Y,D)$ be described at a stratum by components $D_1,\ldots,D_s$ of $D$ and weights $\alpha_1,\ldots,\alpha_s\in\R_{\ge0}$, not all zero.  Set
\[
  A_{X,\Delta}(v_\alpha)
  :=\sum_{i=1}^s\alpha_iA_{X,\Delta}(D_i).
\]
For $v\in\Val_X$, define
\[
  A_{X,\Delta}(v)
  :=\sup_{(Y,D)}
  A_{X,\Delta}\bigl(r_{Y,D}(v)\bigr)
  \in\R\cup\{+\infty\},
\]
where the supremum is taken over the log smooth models dominating $(X,\Delta)$.  We write $A_X(v):=A_{X,0}(v)$.
\end{definition}

\begin{definition}[Local invariants]\label{def:local-invariants}
Let $x\in(X,\Delta)$ be an $n$-dimensional klt germ and let $v\in\Val_{X,x}$.  Its valuation ideals and volume are
\[
  \mathfrak a_m(v)
  :=\{f\in\mathcal O_{X,x}\mid v(f)\ge m\},
  \qquad
  \vol_{X,x}(v)
  :=\lim_{m\to\infty}
  \frac{\ell\bigl(\mathcal O_{X,x}/\mathfrak a_m(v)\bigr)}{m^n/n!},
\]
where $\ell$ denotes length; see \cite{ELS03}.  Define
\[
  \volhat_{X,\Delta}(v)
  :=
  \begin{cases}
    A_{X,\Delta}(v)^n\vol_{X,x}(v),
      & A_{X,\Delta}(v)<+\infty,\\
    +\infty,
      & A_{X,\Delta}(v)=+\infty.
  \end{cases}
\]
The local volume of the germ and the minimal log discrepancy at $x$ are
\[
  \volhat(x,X,\Delta)
  :=\inf_{v\in\Val_{X,x}}\volhat_{X,\Delta}(v),
  \qquad
  \mld_x(X,\Delta)
  :=\inf_F A_{X,\Delta}(F),
\]
where $F$ runs over prime divisors over $X\ni x$.
\end{definition}

We say that klt singularities $x\in X$ and $x'\in X'$ are analytically
isomorphic if $\widehat{\mathcal O}_{X,x}\cong\widehat{\mathcal O}_{X',x'}$.
Then $\volhat(x,X)=\volhat(x',X')$ by \cite[Proposition~2.24]{HLQ23}, and
$\mld_xX=\mld_{x'}X'$ by \cite[Proposition~6]{Mas01}.

By \cite[Lemma~3.5]{CH21}, the infimum defining the minimal log discrepancy is attained by a prime divisor.

The existence of a normalized volume minimizer was proved by Blum over an uncountable ground field \cite[Main Theorem and Section~7]{Blu18} and by Xu over an arbitrary algebraically closed field \cite[Remark~3.8]{Xu20}. Xu also proved quasi-monomiality \cite[Theorem~1.2]{Xu20}, and Xu--Zhuang proved uniqueness up to positive scaling \cite[Theorem~1.1]{XZ21}. For pairs, see \cite[Theorems~3.3 and~3.4]{HLQ23}. 

\begin{theorem}\label{thm:minimizer}
Let $x\in(X,\Delta)$ be a klt germ.  The normalized volume functional
$\volhat_{X,\Delta}$ on $\Val_{X,x}$ admits a quasi-monomial minimizer,
unique up to positive scaling.
\end{theorem}

We use the term ``adapted to a quasi-monomial valuation'' in the sense of \cite[Definition~2.11 and the paragraph preceding Theorem~2.17]{XZ26}.

\subsection{Koll\'ar components and the \texorpdfstring{$\delta$-invariant}{delta-invariant}}

\begin{definition}\label{def:kollar-component}
Let $x\in(X,\Delta)$ be a klt germ.  A \emph{Koll\'ar component} over $x$ is a prime divisor $S$ extracted by a projective birational morphism $\pi\colon(Y,S+\Delta_Y)\to(X,\Delta)$ such that $\operatorname{Exc}(\pi)=S$, $\pi(S)=x$, the divisor $S$ is $\Q$-Cartier, the pair $(Y,S+\Delta_Y)$ is plt, and $-S$ is $\pi$-ample, where $\Delta_Y$ is the strict transform of $\Delta$.  Then
\[
  K_Y+S+\Delta_Y
  =\pi^*(K_X+\Delta)+A_{X,\Delta}(S)S.
\]
By the adjunction formula, we may write
\[
  (K_Y+S+\Delta_Y)|_S=K_S+\Delta_S.
\]
Then $(S,\Delta_S)$ is a klt log Fano pair.
For $p\in S$, let $\ind_p(Y,S)$ denote the Cartier index of $S$ as a Weil divisor on $Y$ at $p$, and put $I(S):=\max_{p\in S}\ind_p(Y,S)$.
\end{definition}

The birational model $Y$ is uniquely determined over $X$ by the divisorial valuation $\ord_S$ \cite[Definition~1.1 and the paragraph following it]{LX20}, so $I(S)$ is well defined.  It is finite because a fixed positive multiple of $S$ is Cartier near the projective exceptional divisor.

\begin{definition}\label{def:delta}
Let $(V,B)$ be a $d$-dimensional log Fano pair and put $L:=-(K_V+B)$.  For a prime divisor $F$ over $V$, choose a model $\mu\colon V'\to V$ on which $F$ appears and set
\[
  S_L(F):=\frac{1}{L^d}\int_0^\infty
  \vol(\mu^*L-tF)\,dt.
\]
The $\delta$-invariant of $(V,B)$ is
\[
  \delta(V,B):=\inf_F\frac{A_{V,B}(F)}{S_L(F)},
\]
where $F$ runs over all prime divisors over $V$.
\end{definition}

For rational coefficients, the Fujita--Li valuative criterion
\cite[Theorem~6.5]{Fuj19} (cf.\ \cite[Theorem~3.7]{Li17}) shows that
$(V,B)$ is K-semistable if and only if $\delta(V,B)\ge1$.  For real
coefficients, we use the valuative definition of K-semistability in
\cite[Definition~2.5]{LZ24}, and the same equivalence follows from
\cite[Theorem~2.6]{LZ24}.

\medskip

The following local-to-global comparison was first proved by Liu for K-semistable $\Q$-Fano varieties \cite[Theorem~2]{Liu18}. For log Fano pairs with rational coefficients, the form involving the $\delta$-invariant follows from \cite[Theorem~D]{BJ20}; it is used in this form in \cite[Proof of Corollary~2.22]{LZ24}.

\begin{lemma}\label{lem:local-global}
Let $(V,B)$ be a $d$-dimensional log Fano pair, and put $L:=-(K_V+B)$.  Then every closed point $p\in V$ satisfies
\[
  \delta(V,B)^dL^d
  \le\left(\frac{d+1}{d}\right)^d\volhat(p,V,B).
\]
\end{lemma}

\begin{proof}
When $B$ has rational coefficients, the assertion is the log-pair form of \cite[Theorem~D]{BJ20}; see \cite[Proof of Corollary~2.22]{LZ24}.

For real coefficients, write $B=\sum_{i=1}^m b_iB_i$ and put
$\mathbf b:=(b_1,\ldots,b_m)$.  For
$\mathbf t=(t_1,\ldots,t_m)$, set
$B(\mathbf t):=\sum_{i=1}^m t_iB_i$.  By
\cite[Propositions~2.10 and~2.15]{LZ24}, there are rational vectors
$\mathbf b_j\to\mathbf b$ in a rational polytope such that every
$(V,B(\mathbf b_j))$ is log Fano and
$\delta(V,B(\mathbf b_j))\to\delta(V,B)$.  The anticanonical volume is polynomial in the coefficients on this polytope, and
$\volhat(p,V,B(\mathbf b_j))\to\volhat(p,V,B)$ by
\cite[Lemma~5.3]{HLQ23}.  Passing to the limit proves the assertion.
\end{proof}

\begin{lemma}\label{lem:kollar-volume}
Let $x\in(X,\Delta)$ be a klt germ and let $S$ be a Koll\'ar component over $x$.  Put $d:=\dim S$ and $L:=-(K_S+\Delta_S)$.  Then
\[
  \volhat_{X,\Delta}(\ord_S)=A_{X,\Delta}(S)L^d.
\]
\end{lemma}

\begin{proof}
Let $\pi\colon Y\to X$ be the plt blow-up extracting $S$, and let $\Delta_Y$ be the strict transform of $\Delta$.  By \cite[Lemma~2.11]{LX20}, $\vol_{X,x}(\ord_S)=(-S|_S)^d$.  Restricting the equality in Definition~\ref{def:kollar-component} to $S$
and using the adjunction formula gives
$L=-A_{X,\Delta}(S)S|_S$.
\end{proof}

\subsection{Finite morphisms and local bounds}

\begin{lemma}\label{lem:index-cover}
Let $p\in V$ be a normal germ and let $D_1,\ldots,D_s$ be $\Q$-Cartier Weil divisors on $V$.  There is a finite quasi-\'etale Galois morphism $g\colon(\widetilde p\in\widetilde V)\to(p\in V)$ with finite abelian Galois group such that $g^{-1}(p)=\{\widetilde p\}$ and each $g^*D_i$ is Cartier. If $s=1$ and $D_1$ has Cartier index $r$ at $p$, then $g$ is cyclic of degree $r$.

If $p\in(V,B)$ is klt and $\widetilde B$ is defined by
$K_{\widetilde V}+\widetilde B=g^*(K_V+B)$, then
$\widetilde p\in(\widetilde V,\widetilde B)$ is klt.
\end{lemma}

\begin{proof}
The assertions about the cover are \cite[Section~2.5, Lemma-Definition~2.15]{Zhu24a}. In the case of one divisor, the subgroup generated by the class of $D_1$ in the local class group is cyclic of order equal to its Cartier index. The final assertion follows from \cite[Proposition~5.20]{KM98}.
\end{proof}

The finite degree formula for finite crepant Galois morphisms was proved by Xu--Zhuang \cite[Theorem~1.3]{XZ21}.  Although it was originally stated for $\Q$-divisors, the same proof applies to the pairs considered here; see the discussion preceding \cite[Theorem~2.15]{HLQ23}.

\begin{lemma}[Finite degree formula]\label{lem:finite-degree}
Let $y\in(Y,\Delta_Y)$ and $x\in(X,\Delta)$ be $n$-dimensional klt germs.  Let $f\colon\bigl(y\in(Y,\Delta_Y)\bigr)\to\bigl(x\in(X,\Delta)\bigr)$ be a finite Galois morphism such that $f^{-1}(x)=\{y\}$ and
$K_Y+\Delta_Y=f^*(K_X+\Delta)$. Then
\[
  \volhat(y,Y,\Delta_Y)=\deg f\volhat(x,X,\Delta).
\]
\end{lemma}

We will use the following local upper bound \cite[Theorem~A.4]{LX19}.

\begin{lemma}\label{lem:local-upper}
Let $p\in(V,B)$ be a $d$-dimensional klt germ.  Then
\[
  \volhat(p,V,B)\le d^d,
\]
with equality if and only if $p\in V$ is smooth and $p\notin\Supp B$.
\end{lemma}

\section{Koll\'ar components and discrepancies}\label{sec:engines}

\subsection{Approximation by Koll\'ar components}

Xu--Zhuang proved the following result for rational boundary coefficients
\cite[Lemmas~2.18 and~2.19]{XZ26}.  We verify below that their argument also
applies to real boundaries.

\begin{lemma}\label{lem:real-cone-nearby}
Let $x\in(X,\Delta)$ be a K-semistable log Fano cone with Reeb vector $\xi$ and a torus $T$ acting on it.  For every $\varepsilon>0$, there is a neighborhood $U$ of $\xi$ in the Reeb cone such that, for every quasi-regular Reeb vector $\eta\in U$, the Koll\'ar component $S_\eta$ corresponding to $\eta$ satisfies
\[
  \delta(S_\eta,\Delta_\eta)\ge1-\varepsilon.
\]
If $\pi_\eta\colon(Y_\eta,S_\eta+\Delta_{Y_\eta})\to(X,\Delta)$ is the corresponding plt blow-up, then $\Delta_\eta$ is determined by
$(K_{Y_\eta}+S_\eta+\Delta_{Y_\eta})|_{S_\eta}=K_{S_\eta}+\Delta_\eta$.
\end{lemma}

\begin{proof}
It is enough to consider $0<\varepsilon<1$.  We follow the proofs of
\cite[Lemmas~2.18 and~2.19]{XZ26}.  We use the cone $S$-invariant
$S(\eta;u)$ with the normalization of \cite[Paragraph~2.15]{XZ26}.

Let $\eta$ be quasi-regular.  For a $T$-invariant quasi-monomial valuation
$w$ on $S_\eta$ and $t>0$, let $w_t$ be the associated valuation on the
cone, as in \cite[(2.8)]{XZ26}.  The calculations in
\cite[(2.9) and~(2.10)]{XZ26} give
\begin{align*}
  A_{X,\Delta}(w_t)
  &=A_{S_\eta,\Delta_\eta}(w)+tA_{X,\Delta}(\eta),\\
  S(\eta;w_t)
  &=S_{-(K_{S_\eta}+\Delta_\eta)}(w)
    +tA_{X,\Delta}(\eta).
\end{align*}
These identities remain valid for real coefficients.  The first follows
from the adjunction formula in the statement and the linearity of log discrepancy on a quasi-monomial
cone, while the second is an identity between the corresponding volume
integrals.

We claim that
\begin{equation}\label{eq:real-cone-delta-criterion}
  \delta(S_\eta,\Delta_\eta)\ge1-\varepsilon
\end{equation}
if and only if
\begin{equation}\label{eq:real-cone-valuative-criterion}
  A_{X,\Delta}(u)\ge(1-\varepsilon)S(\eta;u)
\end{equation}
for every $T$-invariant quasi-monomial valuation $u$ centered at the
vertex.  Suppose first that \eqref{eq:real-cone-delta-criterion} holds.
By the valuative formula for the $\delta$-invariant of a log Fano pair with
real coefficients \cite[Section~2.2]{LZ24},
\[
  A_{S_\eta,\Delta_\eta}(w)
  \ge(1-\varepsilon)S_{-(K_{S_\eta}+\Delta_\eta)}(w)
\]
for every quasi-monomial valuation $w$ on $S_\eta$.  The two identities
above therefore imply \eqref{eq:real-cone-valuative-criterion} for every
valuation of the form $w_t$.  Every $T$-invariant quasi-monomial valuation
on the cone, other than a positive multiple of the Reeb valuation $\eta$,
is of this form \cite[Proof of Lemma~2.18]{XZ26}; for positive multiples
of $\eta$, the required inequality follows from
$A_{X,\Delta}(\eta)=S(\eta;\eta)$ and homogeneity.

Conversely, suppose that \eqref{eq:real-cone-valuative-criterion} holds.
Applying it to $w_t$ and letting $t\to0^+$ in the two identities above, we obtain
\[
  A_{S_\eta,\Delta_\eta}(w)
  \ge(1-\varepsilon)S_{-(K_{S_\eta}+\Delta_\eta)}(w)
\]
for every $T$-invariant quasi-monomial valuation $w$ on $S_\eta$.  If
$\delta(S_\eta,\Delta_\eta)<1-\varepsilon$, then, since
$1-\varepsilon<1$, \cite[Theorem~4.7]{LZ24} gives a $T$-invariant special
divisorial valuation computing $\delta(S_\eta,\Delta_\eta)$, contradicting
the preceding inequality.  This proves the claim.

Since the cone is K-semistable, the proof of
\cite[Theorem~3.10]{XZ21}, equivalently \cite[Theorem~2.16]{XZ26}, gives
\begin{equation}\label{eq:real-cone-K-semistable}
  A_{X,\Delta}(u)\ge S(\xi;u)
\end{equation}
for every $T$-invariant quasi-monomial valuation $u$.  The argument applies to real coefficients; see the real-coefficient discussion in
\cite[Section~2.6 and Proof of Theorem~2.19]{Zhu24b}.

Finally, the proof of \cite[Lemma~2.19]{XZ26} shows, after shrinking a
neighborhood $U$ of $\xi$, that
\[
  S(\xi;u)\ge(1-\varepsilon)S(\eta;u)
\]
for every $\eta\in U$ and every $T$-invariant quasi-monomial valuation
$u$.  This part of the argument uses only homogeneity and the continuity
of log discrepancy and valuation volume on the Reeb cone
\cite[Theorem~2.15(3) and Proposition~2.39]{LX18}, and is unchanged for
real coefficients.  Combining this estimate with
\eqref{eq:real-cone-K-semistable} and the claim proves the lemma.
\end{proof}

Li--Xu proved that a normalized volume minimizer is a limit of rescaled
valuations of Koll\'ar components \cite[Theorem~1.3]{LX20}.  Xu--Zhuang
strengthened this for rational boundary coefficients by controlling all nearby
divisorial valuations and the $\delta$-invariants of the corresponding
components \cite[Theorem~2.17]{XZ26}.  The same argument applies to real
coefficients; we include the details below.

\begin{proposition}\label{prop:kollar-approximation}
Let $x\in(X,\Delta)$ be a klt germ.  Let $v_*$ be a minimizer of $\volhat_{X,\Delta}$, and let
$\mu\colon(W,D)\to X$ be a log smooth model adapted to $v_*$.  For every $\varepsilon>0$, there is a neighborhood $U$ of $v_*$ in $\operatorname{QM}(W,D)$ such that every divisorial valuation in $U$ is proportional to the valuation of a Koll\'ar component $S$ satisfying
\[
  \delta(S,\Delta_S)\ge1-\varepsilon.
\]
\end{proposition}

\begin{proof}
We follow the proof of \cite[Theorem~2.17]{XZ26}. By the stable degeneration theorem \cite[Theorem~1.2]{XZ25} and its
extension to real boundaries \cite[Theorem~2.19]{Zhu24b}, the minimizer
$v_*$ induces a K-semistable
log Fano cone $x_0\in(X_0,\Delta_0)$ with Reeb vector $\xi_*$, where
$X_0:=\Spec\operatorname{gr}_{v_*}\mathcal O_{X,x}$ and $\Delta_0$ is the induced degeneration of $\Delta$.  The associated graded ring is finitely generated by the same theorem.  Write $\Delta=\sum_{j=1}^m a_j\Delta_j$ and let $(\Delta_j)_0$ be the induced degeneration of $\Delta_j$.  Then $\Delta_0=\sum_{j=1}^m a_j(\Delta_j)_0$.

By \cite[Lemma~2.10]{LX18}, after shrinking a neighborhood $U$ of $v_*$ in
$\operatorname{QM}(W,D)$, we have
\[
  \operatorname{gr}_w\mathcal O_{X,x}
  \cong\operatorname{gr}_{v_*}\mathcal O_{X,x}
\]
for every $w\in U$.  Under this identification, $w$ determines a Reeb
vector $\xi_w$ near $\xi_*$, and $w$ is divisorial exactly when $\xi_w$ is
quasi-regular.  Apply Lemma~\ref{lem:real-cone-nearby} to the cone
$x_0\in(X_0,\Delta_0)$ with Reeb vector $\xi_*$ and shrink $U$ so that every $\xi_w$ belongs to the resulting neighborhood of $\xi_*$.

The proof of \cite[Proposition~2.59]{LX18}, applied to the original
minimizer $v_*$ and the original pair $(X,\Delta)$, shows after a further
shrinking of $U$ that every divisorial $w\in U$ is proportional to the
valuation of a Koll\'ar component $S_w$.  The corresponding test configuration
has central fiber $(X_0,\Delta_0)$.  Although \cite[Proposition~2.59]{LX18} was stated for rational
boundaries, the same argument applies to real coefficients; see
\cite[Section~2.6 and Proof of Theorem~2.19]{Zhu24b}.  Hence the log Fano pair
$(S_w,\Delta_{S_w})$ obtained by adjunction is isomorphic to the log Fano pair
attached to the quasi-regular Reeb vector
$\xi_w$ by \cite[Lemma~2.10 and Proof of Theorem~2.17]{XZ26}.
By Lemma~\ref{lem:real-cone-nearby},
$\delta(S_w,\Delta_{S_w})\ge1-\varepsilon$, which proves the proposition.
\end{proof}

\subsection{An inequality for Koll\'ar components}

\begin{proposition}\label{prop:component}
Let $x\in(X,\Delta)$ be an $n$-dimensional klt germ, where $n\ge2$, and let $S$ be a Koll\'ar component over $x$.  Then
\[
  \delta(S,\Delta_S)^{n-1}\volhat_{X,\Delta}(\ord_S)
  \le n^{n-1}\frac{A_{X,\Delta}(S)}{I(S)}
  \le n^{n-1}\mld_x(X,\Delta).
\]
\end{proposition}

\begin{proof}
Let $\pi\colon(Y,S+\Delta_Y)\to(X,\Delta)$ be the plt blow-up extracting $S$, and put $d:=n-1$ and
$L:=-(K_S+\Delta_S)$.  Choose $p\in S$ such that
$\ind_p(Y,S)=I(S)$.  By \cite[Theorem~A.1]{HLS26}, we have
$\ind_p(S,S|_S)=I(S)$.  Apply the one-divisor case of
Lemma~\ref{lem:index-cover} to $S$ on the germ $p\in Y$.
Let $h\colon(\widetilde p\in\widetilde Y)\to(p\in Y)$ be the resulting
degree-$I(S)$ cyclic quasi-\'etale Galois cover, and put
$\widetilde S:=h^{-1}(S)_{\mathrm{red}}$.  By construction, $h^*S$ is
Cartier near $\widetilde p$, and quasi-\'etaleness implies
$h^*S=\widetilde S$.  Define $\widetilde\Delta_Y$ by
\[
  K_{\widetilde Y}+\widetilde S+\widetilde\Delta_Y
  =h^*(K_Y+S+\Delta_Y).
\]
Then $\widetilde\Delta_Y=h^*\Delta_Y\ge0$.  By
\cite[Proposition~5.20(3)]{KM98},
$(\widetilde Y,\widetilde S+\widetilde\Delta_Y)$ is plt.  Every
irreducible component of $h^{-1}(S)_{\mathrm{red}}$ contains
$\widetilde p$, since it maps onto $S$ and
$h^{-1}(p)=\{\widetilde p\}$.  Since the pair is plt, distinct
irreducible components of $\widetilde S$ are disjoint and each component is
normal.  Hence $h^{-1}(S)_{\mathrm{red}}=\widetilde S$ is prime and normal near
$\widetilde p$.  Moreover, $h$ is \'etale at the generic point of $S$.
Thus the induced morphism
$g\colon(\widetilde p\in\widetilde S)\to(p\in S)$ is cyclic of degree
$I(S)$.  By the adjunction formula, we may write
\[
  K_{\widetilde S}+\widetilde\Delta_S
  =g^*(K_S+\Delta_S).
\]
Then $(\widetilde S,\widetilde\Delta_S)$ is klt.
By Lemmas~\ref{lem:finite-degree} and~\ref{lem:local-upper},
\[
  I(S)\volhat(p,S,\Delta_S)
  =\volhat(\widetilde p,\widetilde S,\widetilde\Delta_S)
  \le d^d.
\]
By Lemma~\ref{lem:local-global} and the preceding inequality,
\[
  \delta(S,\Delta_S)^dL^d
  \le\left(\frac{n}{d}\right)^d
  \volhat(p,S,\Delta_S)
  \le\frac{n^d}{I(S)}.
\]
Multiplication by $A_{X,\Delta}(S)$ and Lemma~\ref{lem:kollar-volume} prove the first inequality.

Let $F$ be a prime divisor over $X\ni x$.  If $F\ne S$, then
its center $\eta$ on $Y$ lies in $S$.  Let $m:=\ind_\eta(Y,S)$.  Since
$mS$ is Cartier near $\eta$ and $\ord_F(mS)\in\Z_{>0}$, we have
$\ord_F(S)\ge1/m\ge1/I(S)$.  Since $(Y,S+\Delta_Y)$ is plt,
\[
  A_{X,\Delta}(F)
  =A_{Y,S+\Delta_Y}(F)+A_{X,\Delta}(S)\ord_F(S)
  >\frac{A_{X,\Delta}(S)}{I(S)}.
\]
For $F=S$, $A_{X,\Delta}(S)/I(S)\le A_{X,\Delta}(S)$.  Taking the infimum over $F$ proves the second inequality.
\end{proof}

\section{Proofs of the main results}\label{sec:sharp}

\subsection{The sharp inequality}

Blum proved the following inequality and its equality characterization when $\Delta=0$ \cite[Lemma~4.7]{Blu18}.  The same proof applies to pairs after using the Izumi-type estimate for klt pairs \cite[Lemma~2.14]{HLQ23} to obtain $0<w(\mathfrak a_\bullet(v_*))<+\infty$.

\begin{lemma}\label{lem:valuation-ideal}
Let $x\in(X,\Delta)$ be an $n$-dimensional klt germ, and let $v_*$
be a minimizer of $\volhat_{X,\Delta}$.  For $w\in\Val_{X,x}$, set
$w(\mathfrak a_\bullet(v_*))
  :=\inf_{k\ge1}\frac{w(\mathfrak a_k(v_*))}{k}$. Then
\begin{equation}\label{eq:valuation-ideal}
  A_{X,\Delta}(v_*)
  \le
  \frac{A_{X,\Delta}(w)}{w(\mathfrak a_\bullet(v_*))}.
\end{equation}
Moreover, equality holds if and only if $w$ is proportional to $v_*$.
\end{lemma}

We will use the following lemma in the equality case.

\begin{lemma}\label{lem:mld-volume-minimizer}
Let $x\in(X,\Delta)$ be an $n$-dimensional klt germ, where $n\ge2$, and assume that
\[
  \volhat(x,X,\Delta)=n^{n-1}\mld_x(X,\Delta).
\]
Then there is a unique prime divisor $E$ over $X\ni x$ computing
$\mld_x(X,\Delta)$.  Moreover,
\[
  \volhat_{X,\Delta}(\ord_E)=\volhat(x,X,\Delta).
\]
\end{lemma}

\begin{proof}
\noindent\textit{Step 1.}
Let $E$ be any prime divisor over $X\ni x$ computing $\mld_x(X,\Delta)$, so
\begin{equation}\label{eq:mld-attained}
  A_{X,\Delta}(E)=\mld_x(X,\Delta).
\end{equation}
By Theorem~\ref{thm:minimizer}, let $v_*$ be a minimizer of $\volhat_{X,\Delta}$, and put $d:=n-1$.  Replace the adapted quasi-monomial cone by its smallest face containing $v_*$, so that $v_*$ lies in its relative interior.  By Proposition~\ref{prop:kollar-approximation} and the density of rational rays in this face, there are Koll\'ar components $S_i$, positive numbers $c_i$, and numbers $\tau_i\to0$ such that
\[
  v_i:=c_i\ord_{S_i}\to v_*,
  \qquad
  \delta(S_i,\Delta_{S_i})\ge1-\tau_i.
\]
Since log discrepancy is continuous on the chosen quasi-monomial cone,
$A_{X,\Delta}(v_i)\to A_{X,\Delta}(v_*)$.  Replacing $v_i$ by
$\frac{A_{X,\Delta}(v_*)}{A_{X,\Delta}(v_i)}v_i$, we may assume that
$A_{X,\Delta}(v_i)=A_{X,\Delta}(v_*)$, while still $v_i\to v_*$.  If $S_i=E$ for infinitely many $i$, then along this subsequence
\[
  v_i=\frac{A_{X,\Delta}(v_*)}{A_{X,\Delta}(E)}\ord_E.
\]
Hence $v_*$ is proportional to $\ord_E$, so $\ord_E$ minimizes the normalized
volume.  Otherwise, after discarding finitely many terms, we may assume that
$S_i\ne E$ for every $i$.  It remains to treat this case.

By the equality in the statement, the minimality of $v_*$, and
Proposition~\ref{prop:component}, we have
\begin{equation}\label{eq:component-limit}
\begin{split}
  n^d\mld_x(X,\Delta)\delta(S_i,\Delta_{S_i})^d
  &=\delta(S_i,\Delta_{S_i})^d\volhat_{X,\Delta}(v_*)\\
  &\le\delta(S_i,\Delta_{S_i})^d\volhat_{X,\Delta}(\ord_{S_i})\\
  &\le n^d\frac{A_{X,\Delta}(S_i)}{I(S_i)}
  \le n^d\mld_x(X,\Delta).
\end{split}
\end{equation}
In particular, $\delta(S_i,\Delta_{S_i})\le1$.  Since
$\delta(S_i,\Delta_{S_i})\ge1-\tau_i$, we have
$\delta(S_i,\Delta_{S_i})\to1$, and \eqref{eq:component-limit} gives
\[
  \frac{A_{X,\Delta}(S_i)}{I(S_i)}\to\mld_x(X,\Delta).
\]
For each $i$, let $\pi_i\colon Y_i\to X$ be the plt blow-up extracting $S_i$, and let $\Delta_{Y_i}$ be the strict transform of $\Delta$.  Then
\[
  \mld_x(X,\Delta)
  =A_{Y_i,S_i+\Delta_{Y_i}}(E)
   +A_{X,\Delta}(S_i)\ord_E(S_i).
\]
The first term is positive, and
$\ord_E(S_i)\ge1/I(S_i)$.  Hence
\[
  \frac{A_{X,\Delta}(S_i)}{I(S_i)}
  \le A_{X,\Delta}(S_i)\ord_E(S_i)
  <\mld_x(X,\Delta).
\]
Together with the preceding inequalities, this gives
\begin{equation}\label{eq:product-limit}
  A_{X,\Delta}(S_i)\ord_E(S_i)
  \to\mld_x(X,\Delta).
\end{equation}

\smallskip
\noindent\textit{Step 2.}
Put $\mathfrak a_k:=\mathfrak a_k(v_*)$ and
$\mathfrak a_\bullet:=\mathfrak a_\bullet(v_*)$ as in
Lemma~\ref{lem:valuation-ideal}.  Since $v_*$ lies in the relative interior of the chosen face and
$v_i\to v_*$, there are numbers $\varepsilon_i\to0$ such that
$v_i(f)\ge(1-\varepsilon_i)v_*(f)$ for every $f\in\mathcal O_{X,x}$.  Hence
$v_i(\mathfrak a_\bullet)\ge1-\varepsilon_i$.  By
\eqref{eq:valuation-ideal} and
$A_{X,\Delta}(v_i)=A_{X,\Delta}(v_*)$, we have
$v_i(\mathfrak a_\bullet)\le1$.  Consequently,
\begin{equation}\label{eq:ideal-limit}
 \lim_{i\to +\infty} \frac{A_{X,\Delta}(S_i)}{\ord_{S_i}(\mathfrak a_\bullet)}
  =\lim_{i\to +\infty}\frac{A_{X,\Delta}(v_i)}{v_i(\mathfrak a_\bullet)}
  =A_{X,\Delta}(v_*).
\end{equation}

For every $f\in\mathcal O_{X,x}$, we have
$\ord_E(f)\ge\ord_{S_i}(f)\ord_E(S_i)$, and hence
$\ord_E(\mathfrak a_\bullet)\ge
\ord_E(S_i)\ord_{S_i}(\mathfrak a_\bullet)$.  Therefore
\[
\begin{aligned}
  \frac{A_{X,\Delta}(E)}{\ord_E(\mathfrak a_\bullet)}
  &\le \lim_{i\to +\infty}
  \frac{A_{X,\Delta}(E)}
       {\ord_E(S_i)\ord_{S_i}(\mathfrak a_\bullet)}\\
  &=\lim_{i\to +\infty}
  \left(\frac{\mld_x(X,\Delta)}
       {A_{X,\Delta}(S_i)\ord_E(S_i)}\cdot
  \frac{A_{X,\Delta}(S_i)}
       {\ord_{S_i}(\mathfrak a_\bullet)}\right)
  = A_{X,\Delta}(v_*).
\end{aligned}
\]
The reverse inequality is \eqref{eq:valuation-ideal} with $w=\ord_E$.
Thus equality holds in \eqref{eq:valuation-ideal}.  By the equality
characterization in Lemma~\ref{lem:valuation-ideal}, $\ord_E$ is
proportional to $v_*$ and therefore minimizes the normalized volume.

Thus, in either case, $\ord_E$ minimizes the normalized volume.  Since $E$ was
arbitrary, every prime divisor computing $\mld_x(X,\Delta)$ has this property.
If $F$ is another prime divisor over $X\ni x$ computing
$\mld_x(X,\Delta)$, Theorem~\ref{thm:minimizer} gives
$\ord_F=c\ord_E$ for some $c>0$.  Then
$0<A_{X,\Delta}(F)=A_{X,\Delta}(E)
=A_{X,\Delta}(c\ord_E)=cA_{X,\Delta}(E)$,
so $c=1$.  Hence $F=E$.
\end{proof}

The following proposition identifies the equality case once such a divisor is available.

\begin{proposition}\label{prop:equality-case}
Let $x\in(X,\Delta)$ be an $n$-dimensional klt germ, where $n\ge2$, and assume that
\[
  \volhat(x,X,\Delta)=n^{n-1}\mld_x(X,\Delta).
\]
Let $E$ be a prime divisor over $X\ni x$ such that
\[
  A_{X,\Delta}(E)=\mld_x(X,\Delta),
  \qquad
  \volhat_{X,\Delta}(\ord_E)=\volhat(x,X,\Delta).
\]
Then $E$ is a Koll\'ar component isomorphic to $\mathbb P^{n-1}$,
$\Delta=0$ near $x$, and analytically,
$(x\in X)\cong\frac{1}{r}(1,\ldots,1)$ for some integer $r\ge1$.
\end{proposition}

\begin{proof}
\noindent\textit{Step 1.}
Put $d:=n-1$.  For every $\varepsilon>0$, apply
Proposition~\ref{prop:kollar-approximation} to $\ord_E$.  Since $\ord_E$
itself belongs to the resulting neighborhood, it is proportional to the
valuation of a Koll\'ar component.  Proportional divisorial valuations
determine the same prime divisor.  Hence $E$ is a Koll\'ar component and
$\delta(E,\Delta_E)\ge1-\varepsilon$.  Letting $\varepsilon\to0$, we have
\begin{equation}\label{eq:delta-lower}
  \delta(E,\Delta_E)\ge1.
\end{equation}

Write $\pi\colon(Y,E+\Delta_Y)\to(X,\Delta)$ for the plt blow-up
extracting $E$ and put $L:=-(K_E+\Delta_E)$.  The two equalities in the statement, Lemma~\ref{lem:kollar-volume}, and
equality in \eqref{eq:main} give
\[
  A_{X,\Delta}(E)L^d
  =\volhat_{X,\Delta}(\ord_E)
  =\volhat(x,X,\Delta)
  =n^dA_{X,\Delta}(E).
\]
Hence $L^d=n^d$.
Applying Proposition~\ref{prop:component} to $E$ and using \eqref{eq:delta-lower}, we have
\[
  1\le\delta(E,\Delta_E)^d\le\frac{1}{I(E)}\le1.
\]
Hence
\begin{equation}\label{eq:index-delta}
  I(E)=1,
  \qquad
  \delta(E,\Delta_E)=1.
\end{equation}

\smallskip
\noindent\textit{Step 2.}
For every closed point $p\in E$, Lemmas~\ref{lem:local-global} and~\ref{lem:local-upper} and \eqref{eq:index-delta} give
\[
  n^d
  \le\left(\frac{n}{d}\right)^d\volhat(p,E,\Delta_E)
  \le n^d.
\]
Thus $\volhat(p,E,\Delta_E)=d^d$ for every closed point $p\in E$.
By Lemma~\ref{lem:local-upper}, $E$ is smooth and $\Delta_E=0$.  By
\eqref{eq:index-delta}, $\delta(E,0)=1$ and $(-K_E)^d=(d+1)^d$.  In particular,
$(E,0)$ is K-semistable.  By \cite[Theorem~36]{Liu18}, $E\cong\mathbb P^d$.

\smallskip
\noindent\textit{Step 3.}
Since $I(E)=1$, the divisor $E$ is Cartier near $E$.  At a point
$p\in E$, write $E=(u=0)$ in the local ring $\mathcal O_{Y,p}$.  The quotient
$\mathcal O_{Y,p}/(u)=\mathcal O_{E,p}$ is regular.  If $u\in\mathfrak m_p^2$, then
$\operatorname{edim}\mathcal O_{Y,p}=\operatorname{edim}\mathcal O_{E,p}=d$, contradicting
$\dim\mathcal O_{Y,p}=d+1$.  Thus $u\notin\mathfrak m_p^2$ and
$\operatorname{edim}\mathcal O_{Y,p}=d+1$; hence $Y$ is smooth along $E$.

By the adjunction formula,
\[
  (K_Y+E+\Delta_Y)|_E=K_E+\Delta_E,
  \qquad
  (K_Y+E)|_E=K_E.
\]
Hence $\Delta_Y|_E=\Delta_E=0$, so no component of $\Delta_Y$ meets $E$.  Every component of $\Delta$ through $x$ would have strict transform meeting the exceptional fiber, so $\Delta=0$ near $x$.  Write $-E|_E\cong\mathcal O_{\mathbb P^d}(r)$, where $r\ge1$.
By \cite[Section~3]{LX20}, the valuation $\ord_E$ induces a flat
$\mathbb G_m$-equivariant degeneration whose restriction over $\mathbb G_m$ is the product of the germ $x\in X$ with $\mathbb G_m$ and whose central fiber is
\begin{equation}\label{eq:quotient-cone}
  \Spec\bigoplus_{j\ge0}
  H^0\bigl(\mathbb P^d,\mathcal O_{\mathbb P^d}(jr)\bigr)
  \cong\mathbb A^n/\mu_r.
\end{equation}

When $r=1$, the central fiber is smooth, so openness of smoothness and the
product description over $\mathbb G_m$ imply that $x\in X$ is smooth.  Assume
that $r>1$.  If $n\ge3$, the central fiber is an isolated quotient
singularity.  By Schlessinger's rigidity theorem \cite[Theorem~3]{Sch71},
the original germ is analytically isomorphic to the central fiber; cf.\
\cite[Proof of Proposition~4.5]{Liu25}.  Hence \eqref{eq:quotient-cone}
identifies $(x\in X)$ analytically with $\frac{1}{r}(1,\ldots,1)$.

Suppose that $n=2$.  Then $E\cong\mathbb P^1$ and $E^2=-r$.
Since $r>1$ and $E$ is the only exceptional curve, $\pi$ is the minimal
resolution near $E$.  By \cite[Proposition~4.18]{KM98}, the germ $x\in X$
is a quotient surface singularity.  The classification in \cite{Bri68}
shows that it is cyclic, say of type $\frac{1}{m}(1,a)$, and that its dual
resolution graph is determined by the Hirzebruch--Jung continued fraction
$m/a$.  Since the graph consists of a single vertex of weight $-r$, the
Hirzebruch--Jung continued fraction $m/a$ has the single entry $r$.  Thus
$m/a=r$.  Since $\gcd(m,a)=1$, we have $a=1$ and $m=r$, so
$(x\in X)\cong\frac{1}{r}(1,1)$.
\end{proof}

\begin{proof}[Proof of Theorem~\ref{thm:main}]
Put $d:=n-1$, and let $v_*$ be a minimizer of $\volhat_{X,\Delta}$.  Fix $0<\varepsilon<1$ and choose a divisorial valuation sufficiently close to $v_*$ in an adapted quasi-monomial cone.  By Proposition~\ref{prop:kollar-approximation}, it is proportional to the valuation of a Koll\'ar component $S$ satisfying
$\delta(S,\Delta_S)\ge1-\varepsilon$.  Since $\volhat_{X,\Delta}(v_*)=\volhat(x,X,\Delta)$,
Proposition~\ref{prop:component} gives
\[
  \volhat(x,X,\Delta)
  \le\volhat_{X,\Delta}(\ord_S)
  \le\frac{n^d}{(1-\varepsilon)^d}\mld_x(X,\Delta).
\]
Letting $\varepsilon\to0$ proves \eqref{eq:main}.

Suppose that equality holds.  The asserted equality statement follows from
Lemma~\ref{lem:mld-volume-minimizer} and Proposition~\ref{prop:equality-case}.

Conversely, assume that $\Delta=0$ near $x$ and analytically
$(x\in X)\cong\frac{1}{r}(1,\ldots,1)$.  It suffices to compute on
$0\in\mathbb A^n/\mu_r$, where a generator of $\mu_r$ acts diagonally with
weights $(1,\ldots,1)$.  By Lemma~\ref{lem:finite-degree},
$\volhat(x,X)=n^n/r$.  Put
$N_r:=\Z^n+\Z\frac{1}{r}(1,\ldots,1)$.  By
\cite[Section~1]{Amb06},
\[
  \mld_x(X)
  =\min_{v=(v_1,\ldots,v_n)\in N_r\cap\R_{>0}^n}
    \sum_{i=1}^n v_i
  =\frac{n}{r}.
\]
Hence equality in \eqref{eq:main} holds.
\end{proof}

\begin{proof}[Proof of Corollary~\ref{cor:mld-one}]
By Theorem~\ref{thm:main},
\[
  \volhat(x,X,\Delta)
  \le n^{n-1}\mld_x(X,\Delta)
  \le n^{n-1}.
\]
If equality holds, then $\mld_x(X,\Delta)=1$.  By Theorem~\ref{thm:main},
$\Delta=0$ near $x$ and analytically
$(x\in X)\cong\frac{1}{r}(1,\ldots,1)$.  By the toric calculation in the proof of Theorem~\ref{thm:main},
$1=\mld_x(X)=\frac{n}{r}$, so $r=n$.
Conversely, if $\Delta=0$ near $x$ and analytically
$(x\in X)\cong\frac{1}{n}(1,\ldots,1)$, then
$\mld_x(X)=1$ and $\volhat(x,X)=n^{n-1}$, so equality in
Theorem~\ref{thm:main} holds.
\end{proof}

\subsection{Discreteness}

We use the following consequence of the uniform lc rational polytope theorem
\cite[Theorem~5.6 and Lemma~5.4]{HLS26} and
Lemma~\ref{lem:index-cover}.

\begin{lemma}\label{lem:discrepancy-finiteness}
Fix an integer $n\ge2$ and a finite set $\Gamma\subset[0,1]$.  There is a set $\Lambda\subset\R_{\ge0}$, depending only on $n$ and $\Gamma$, such that $\Lambda\cap[0,M]$ is finite for every $M>0$ and the following property holds.  For every $n$-dimensional klt germ $x\in(X,\Delta)$ satisfying $\operatorname{Coeff}(\Delta)\subseteq\Gamma$, there is a finite quasi-\'etale abelian Galois morphism $f\colon\bigl(y\in(Y,\Delta_Y)\bigr)\to\bigl(x\in(X,\Delta)\bigr)$ such that $f^{-1}(x)=\{y\}$,
\[
  K_Y+\Delta_Y=f^*(K_X+\Delta),
  \qquad
  \operatorname{Coeff}(\Delta_Y)\subseteq\Gamma,
\]
and, if $N:=\deg f$, then
\[
  N\mld_x(X,\Delta)\in\Lambda\setminus\{0\}.
\]
\end{lemma}

\begin{proof}
List the positive elements of $\Gamma$ as $\gamma_1,\ldots,\gamma_t$, and write
$\Delta=\sum_{k=1}^t\gamma_kB_k$, where the $B_k$ are reduced and have no common irreducible components.  If $t=0$, take $s=q=1$, $\lambda_1=1$, and $\Delta^{(1)}=0$.  Otherwise, let $V\subset\R^t$ be the smallest affine subspace defined over $\Q$ that contains $(\gamma_1,\ldots,\gamma_t)$.  By \cite[Theorem~5.6 and Lemma~5.4]{HLS26}, one may choose rational points in a fixed neighborhood of $(\gamma_1,\ldots,\gamma_t)$ in $V$ whose convex hull contains $(\gamma_1,\ldots,\gamma_t)$ in its relative interior.  Thus there are positive real numbers $\lambda_1,\ldots,\lambda_s$ with sum one, an integer $q>0$, and lc $\Q$-boundaries $\Delta^{(1)},\ldots,\Delta^{(s)}$ such that
\[
  \Delta=\sum_{j=1}^s\lambda_j\Delta^{(j)},
  \qquad
  q(K_X+\Delta^{(j)})
  \text{ is an integral $\Q$-Cartier Weil divisor}
\]
for every $j$.  The numbers $s,q,\lambda_1,\ldots,\lambda_s$ depend only on $n$ and $\Gamma$.

Apply Lemma~\ref{lem:index-cover} simultaneously to the divisors $q(K_X+\Delta^{(j)})$.  Let $f\colon(y\in Y)\to(x\in X)$ be the resulting cover and put $N:=\deg f$.  Define $\Delta_Y^{(j)}$ and $\Delta_Y$ by $K_Y+\Delta_Y^{(j)}=f^*(K_X+\Delta^{(j)})$ and $K_Y+\Delta_Y=f^*(K_X+\Delta)$.  Then $\Delta_Y=\sum_j\lambda_j\Delta_Y^{(j)}$, each $q(K_Y+\Delta_Y^{(j)})$ is Cartier, and the coefficients of $\Delta_Y$ belong to $\Gamma$.

Choose a prime divisor $E$ computing $\mld_x(X,\Delta)$ and a prime divisor $G$ over $Y$ lying over $E$.  If $e$ is its ramification index, then $e\mid N$.  Since each $(Y,\Delta_Y^{(j)})$ is lc, put $b_j:=qA_{Y,\Delta_Y^{(j)}}(G)\in\Z_{\ge0}$.  By \cite[Proposition~5.20]{KM98},
$A_{X,\Delta^{(j)}}(E)=b_j/(qe)$.  Since
$\Delta=\sum_j\lambda_j\Delta^{(j)}$, we have
\[
  \begin{aligned}
  N\mld_x(X,\Delta)
  &=N A_{X,\Delta}(E)\\
  &=N\sum_{j=1}^s\lambda_jA_{X,\Delta^{(j)}}(E)\\
  &=\frac{1}{q}\sum_{j=1}^s\lambda_j\frac{N}{e}\,b_j.
  \end{aligned}
\]
Thus we may take $\Lambda:=\sum_{j=1}^s\frac{\lambda_j}{q}\Z_{\ge0}$.  Since the generators $\lambda_j/q$ are positive, $\Lambda\cap[0,M]$ is finite for every $M>0$.
\end{proof}

For $c>0$, we use the following notation.

\begin{definition}\label{def:ratio-c}
For an integer $n\ge2$, a set $\Gamma\subset[0,1]$, and a real number
$c>0$, set
\[
  \mathcal R^{(c)}_{n,\Gamma}
  :=\left\{
  \frac{\volhat(x,X,\Delta)}{\mld_x(X,\Delta)^c}
  \ \middle|\ 
  \begin{array}{c}
  x\in(X,\Delta)\text{ is an }n\text{-dimensional klt germ},\\[-2pt]
  \operatorname{Coeff}(\Delta)\subseteq\Gamma
  \end{array}
  \right\}.
\]
Thus $\mathcal R^{(1)}_{n,\Gamma}=\mathcal R_{n,\Gamma}$.
\end{definition}

\begin{proposition}\label{prop:acc-threshold}
Fix an integer $n\ge2$ and a finite set $\Gamma\subset[0,1]$.  For every
$0<c\le1$ and every $\varepsilon>0$, the set
\[
  \mathcal R^{(c)}_{n,\Gamma}\cap[\varepsilon,\infty)
\]
is finite.  For every $c>1$, the set $\mathcal R^{(c)}_{n,\Gamma}$ is
unbounded and hence does not satisfy the ACC.  Thus $c=1$ is the largest
positive exponent for which the ACC holds.
\end{proposition}

\begin{proof}
Fix $0<c\le1$ and $\varepsilon>0$, and let $x\in(X,\Delta)$ be an
$n$-dimensional klt germ with $\operatorname{Coeff}(\Delta)\subseteq\Gamma$
such that $\volhat(x,X,\Delta)/\mld_x(X,\Delta)^c\ge\varepsilon$.
Apply Lemma~\ref{lem:discrepancy-finiteness}, and let $f\colon\bigl(y\in(Y,\Delta_Y)\bigr)\to\bigl(x\in(X,\Delta)\bigr)$ be the resulting morphism.  Put $N:=\deg f$ and
$a:=N\mld_x(X,\Delta)\in\Lambda\setminus\{0\}$.  By
Lemma~\ref{lem:finite-degree},
\[
  \frac{\volhat(x,X,\Delta)}{\mld_x(X,\Delta)^c}
  =\frac{\volhat(y,Y,\Delta_Y)}{a^cN^{1-c}}.
\]
Since $\volhat(y,Y,\Delta_Y)\le n^n$, we have
$a^cN^{1-c}\le\frac{n^n}{\varepsilon}$.  As $N\ge1$, this gives $a\le(n^n/\varepsilon)^{1/c}$.  Since
$\Lambda\cap[0,(n^n/\varepsilon)^{1/c}]$ is finite, only finitely many values of $a$ occur.  If $0<c<1$, then
for each such $a$ the displayed inequality bounds the positive integer $N$.
If $c=1$, the factor $N^{1-c}$ is equal to one.  Thus in either case
$a^cN^{1-c}$ takes only finitely many values.  For each of these values,
$\varepsilon a^cN^{1-c}\le\volhat(y,Y,\Delta_Y)\le n^n$.
The coefficients of $\Delta_Y$ belong to $\Gamma$.  By
\cite[Theorem~1.2]{XZ24}, only finitely many local volumes occur in each
such interval.  Hence
$\mathcal R^{(c)}_{n,\Gamma}\cap[\varepsilon,\infty)$ is finite.

Now let $c>1$.  For $r\ge1$, set
$(0\in X_r):=\frac{1}{r}(1,\ldots,1)$.  By the computation in the proof of Theorem~\ref{thm:main},
\[
  \volhat(0,X_r)=\frac{n^n}{r},
  \qquad
  \mld_0(X_r)=\frac{n}{r}.
\]
Consequently,
\[\frac{\volhat(0,X_r)}{\mld_0(X_r)^c}
  =n^{n-c}r^{c-1},\]
which tends to $+\infty$ as $r\to\infty$.
\end{proof}

\begin{proof}[Proof of Theorem~\ref{thm:ratio-acc}]
Take $c=1$ in Proposition~\ref{prop:acc-threshold}.
\end{proof}

\begin{remark}\label{rem:semicontinuity}
The local volume is lower semicontinuous in $\Q$-Gorenstein flat
families of klt singularities \cite[Theorem~1]{BL21}, but the quotient
$\volhat/\mld$ is neither lower nor upper semicontinuous in general.
For upper semicontinuity, consider the flat hypersurface family
\[
  (0\in X_t):=\bigl(x_1^2+x_2^2+x_3^2+x_4^2+tx_4=0\bigr).
\]
For $t\ne0$, the germ $0\in X_t$ is smooth, while $0\in X_0$ is a
three-dimensional ordinary double point.  The corresponding values of
$\volhat/\mld$ are $27/3=9$ and $16/2=8$, respectively; see
Lemma~\ref{lem:local-upper} and \cite[Example~5.2]{Liu25}.

For lower semicontinuity, one may take
\[
  (0\in Y_t):=\bigl(x^2+y^3+u^6+v^7+tu^5=0\bigr).
\]
The special fiber is of Brieskorn--Pham type $(2,3,6,7)$, while for
$t\ne0$ the germ is analytically of type $(2,3,5,7)$.  The local volumes are
given by \cite[Theorem~1.4 and Example~2.12]{LST25}.  Using Ambro's inversion
of adjunction for nondegenerate hypersurfaces \cite[Main Theorem]{Amb03}, one
computes the minimal log discrepancies.  Thus
\[
  \bigl(\mld_0(Y_0),\volhat(0,Y_0)\bigr)
  =\left(1,\frac{36}{49}\right),
  \qquad
  \bigl(\mld_0(Y_t),\volhat(0,Y_t)\bigr)
  =\left(2,\frac{37^3}{210^2}\right)
\]
for $t\ne0$.  Since
$36/49>37^3/(2\cdot210^2)$, the quotient is not lower semicontinuous.

A stable degeneration associated with a normalized volume minimizer
preserves the local volume \cite[Theorem~1.2]{XZ25}.  Thus a strict drop
of the minimal log discrepancy in such a degeneration gives another
source of failure of lower semicontinuity.
\end{remark}

\section{Applications and questions}\label{sec:applications}

\subsection{Minimal log discrepancies of log Fano pairs}

\begin{proof}[Proof of Corollary~\ref{cor:kmoduli}]
By Lemma~\ref{lem:local-global} and Theorem~\ref{thm:main},
\[
  \left(\frac{n}{n+1}\right)^n
  \delta(X,\Delta)^n\bigl(-(K_X+\Delta)\bigr)^n
  \le \volhat(x,X,\Delta)
  \le n^{n-1}\mld_x(X,\Delta).
\]
This proves \eqref{eq:kmoduli-mld}.  If
$(X,\Delta)$ is K-semistable, then $\delta(X,\Delta)\ge1$ by
\cite[Theorem~2.6]{LZ24}.
\end{proof}

\subsection{Threefold log Fano pairs}

Liu classified the threefold singularities needed below in \cite[Theorem~1.2 and Propositions~4.3--4.4]{Liu25}.  We apply these results to the underlying variety by comparing $\volhat(x,X)$ with $\volhat(x,X,\Delta)$.  The assumption that $X$ is $\Q$-Gorenstein is needed to define $\volhat(x,X)$ and the index-one cover of $K_X$.

\begin{corollary}\label{cor:threefold-boundary}
Let $(X,\Delta)$ be a three-dimensional log Fano pair, and assume that $X$ is $\Q$-Gorenstein.  Set $V:=\bigl(-(K_X+\Delta)\bigr)^3$ and
$V_\delta:=\delta(X,\Delta)^3V$.  Then the following statements hold, with the terminology of
\cite[Definitions~2.12--2.15]{Liu25}.

\begin{enumerate}
\item If $V_\delta\ge26$, then every singular point of $X$ is either a
$\mathrm{cA}_1$ singularity or a quotient singularity of type
$\frac{1}{2}(1,1,1)$.
\item If $V_\delta\ge22$, then every singular point of $X$ is a
$\mathrm{cA}_1$ singularity, an isolated $\mathrm{cA}_2$ singularity, a
$D_\infty$ singularity, or a quotient singularity of type
$\frac{1}{2}(1,1,1)$.
\item If $V_\delta\ge11$, then every non-Gorenstein point of $X$ is a cyclic
quotient of a possibly smooth $\mathrm{cA}_{\le2}$ hypersurface
singularity.
\end{enumerate}

If $X$ is $\Q$-Gorenstein smoothable, then the quotient singularity
$\frac{1}{2}(1,1,1)$ does not occur.  If $(X,\Delta)$ is K-semistable, the same
conclusions hold with $V$ in place of $V_\delta$ in the three numerical
hypotheses.
\end{corollary}

\begin{proof}
Let $x\in X$ be a closed point.  Since $\Delta\ge0$, we have
$A_X(v)\ge A_{X,\Delta}(v)$ for every $v\in\Val_{X,x}$.
Together with Lemma~\ref{lem:local-global}, we have
\[
  \volhat(x,X)
  \ge\volhat(x,X,\Delta)
  \ge\frac{27}{64}V_\delta.
\]

If $V_\delta\ge26$, then
$\volhat(x,X)\ge351/32>32/3$, so (1) follows from
\cite[Theorem~1.2 and Proposition~4.3]{Liu25}.  If $V_\delta\ge22$, then
$\volhat(x,X)\ge297/32>9$, so (2) follows from
\cite[Theorem~1.2 and Proposition~4.4]{Liu25}.

Assume that $V_\delta\ge11$ and that $x\in X$ is non-Gorenstein.  Then
$\volhat(x,X)\ge297/64>9/2$.  Let $(\widetilde x\in\widetilde X)\to(x\in X)$ be the index-one cover of $K_X$ of degree $r\ge2$.  By Lemma~\ref{lem:finite-degree},
$\volhat(\widetilde x,\widetilde X)=r\volhat(x,X)>9$.
Now (3) follows from \cite[Proof of Theorem~1.3(3)]{Liu25}, using
\cite[Theorems~1.1 and~1.2]{Liu25}.  The smoothability assertion follows
from Schlessinger's rigidity theorem \cite[Theorem~3]{Sch71}; see also
\cite[Proof of Theorem~1.3]{Liu25}.  Finally, if $(X,\Delta)$ is
K-semistable, then $\delta(X,\Delta)\ge1$, and hence $V_\delta\ge V$.
\end{proof}

The main inequality also restricts the order of the boundary along divisors
centered at the given point.

\begin{corollary}\label{cor:boundary-restriction}
Let $(X,\Delta)$ be a three-dimensional log Fano pair, assume that $X$ is $\Q$-Gorenstein, and set $V:=\bigl(-(K_X+\Delta)\bigr)^3$ and
$V_\delta:=\delta(X,\Delta)^3V$.  Let $x\in X$ be a closed point and let $F$ be a prime divisor over $X\ni x$.  Then
\begin{equation}\label{eq:boundary-order}
  \ord_F\Delta
  \le A_X(F)-\frac{3}{64}V_\delta.
\end{equation}
In particular, the following statements hold.
\begin{enumerate}
\item If $x\in X$ is smooth, then
\[
  \operatorname{mult}_x\Delta
  \le3-\frac{3}{64}V_\delta.
\]
\item If $x\in X\cong\frac{1}{r}(1,1,1)$ analytically, then
$rV_\delta\le64$.  If equality holds, then $\Delta=0$ near $x$.
\end{enumerate}
\end{corollary}

\begin{proof}
By Corollary~\ref{cor:kmoduli},
\[
  A_X(F)-\ord_F\Delta
  =A_{X,\Delta}(F)
  \ge\mld_x(X,\Delta)
  \ge\frac{3}{64}V_\delta,
\]
which proves \eqref{eq:boundary-order}.  If $x$ is smooth, let $F$ be the
exceptional divisor of the ordinary blow-up.  Then $A_X(F)=3$ and
$\ord_F\Delta=\operatorname{mult}_x\Delta$, so (1) follows.

Suppose that analytically
$(x\in X)\cong\frac{1}{r}(1,1,1)$.  By \cite[Section~1]{Amb06},
$\mld_x(X)=\mld_0(\mathbb A^3/\mu_r)=\frac{3}{r}$.  Let $F$ be a prime
divisor computing $\mld_x(X)$.  Applying \eqref{eq:boundary-order} to $F$ gives
$\frac{3}{64}V_\delta\le\frac{3}{r}$, hence $rV_\delta\le64$.

Assume that $rV_\delta=64$.  Applying \eqref{eq:boundary-order} to $F$ gives
$0\le\ord_F\Delta\le\frac{3}{r}-\frac{3}{64}V_\delta=0$.
Every component of the effective boundary passing through $x$ has positive
order along this divisor.  After replacing $X$ by a sufficiently small
neighborhood of $x$, we may assume that every component of $\Delta$ contains
$x$.  Thus $\Delta=0$ near $x$.
\end{proof}

\subsection{Questions}

For DCC coefficient sets, the set of local volumes satisfies the ACC
\cite[Theorem~1.2]{HLQ24}.  This leads to the following question.

\begin{question}\label{ques:dcc-ratio}
Fix an integer $n\ge2$ and a DCC set $\Gamma\subset[0,1]$.  Does
$\mathcal R_{n,\Gamma}$ satisfy the ACC?
\end{question}

Theorem~\ref{thm:ratio-acc} implies that the second-largest value of
$\mathcal R_{n,\{0\}}$ exists.  We expect the following explicit value.

\begin{question}\label{ques:second-ratio}
For every integer $n\ge2$, the second-largest value of $\mathcal R_{n,\{0\}}$
is
\[
  \begin{cases}
    \dfrac{4}{3}, & n=2,\\[4pt]
    2(n-1)^{n-1}, & n=3,4,\\[4pt]
    \dfrac{n^n}{n+1}, & n\ge5.
  \end{cases}
\]
\end{question}

It is also natural to ask whether the following weighted sum is uniformly
bounded.

\begin{question}\label{ques:discrepancy-upper}
Fix an integer $n\ge2$ and a finite set $\Gamma\subset[0,1]$.  Does there exist
a constant $C_{n,\Gamma}<+\infty$ such that every $n$-dimensional klt germ
$x\in(X,\Delta)$ with $\operatorname{Coeff}(\Delta)\subseteq\Gamma$
satisfies
\[
  \volhat(x,X,\Delta)
  \sum_{\substack{
    E\text{ a prime divisor over }X\ni x\\
    A_{X,\Delta}(E)<\mld_x(X,\Delta)+1}}
  \frac{1}{A_{X,\Delta}(E)}
  \le C_{n,\Gamma}.
\]
\end{question}

\end{document}